\documentclass[12pt,a4paper,oneside]{amsart}

\usepackage[a4paper]{geometry}
\usepackage{mathptmx}
\DeclareMathAlphabet{\mathcal}{OMS}{cmsy}{m}{n} 

\usepackage{enumerate,perpage,stmaryrd}

\newtheorem{theorem}{Theorem}
\newtheorem{proposition}{Proposition}
\newtheorem{lemma}{Lemma}
\newtheorem{corollary}{Corollary}

\DeclareMathOperator{\Der}{Der}
\DeclareMathOperator{\dcobound}{d}
\DeclareMathOperator{\GDer}{GDer}
\DeclareMathOperator{\Hom}{Hom}
\DeclareMathOperator{\id}{id}

\def\liebrack  {\ensuremath{[\,\cdot\, , \cdot\,]}}

\MakePerPage[2]{footnote}

\begin{document}

\title{On contact brackets on the tensor product}
\author{Pasha Zusmanovich}
\address{University of Ostrava, Czech Republic}
\email{pasha.zusmanovich@osu.cz}

\date{First written November 15, 2020; last minor revision November 12, 2022}
\thanks{Lin. Multilin. Algebra, to appear}

\begin{abstract}
We study the behavior of contact brackets on the tensor product of two algebras,
in particular, address the question of Mart\'inez and Zelmanov about extension
of a contact bracket on the tensor product from the brackets on the factors. 
\end{abstract}

\maketitle

\section{Introduction}

We start by recalling some stuff from a recent interesting survey 
\cite{martinez-zelm}. The standing assumptions are that the ground field $K$ is
of characteristic $\ne 2$, and all commutative associative algebras under 
consideration contain a unit $1$. Let $A$ be a commutative associative algebra 
over $K$. A bilinear map $\liebrack: A \times A \to A$ is called a 
\emph{bracket} on $A$. A bracket on $A$ is called \emph{contact}, if 
$(A, \liebrack)$ is a Lie algebra, and the identity
\begin{equation}\label{eq-abc}
[ab,c] = [a,c]b + [b,c]a + [c,1]ab
\end{equation}
holds for any $a,b,c \in A$ (the condition that the linear map $a \mapsto [a,1]$
is a derivation of $A$, included in the definition of the contact bracket in
\cite{martinez-zelm}, follows from the identity (\ref{eq-abc}); in some 
literature such structures appear under the names of \emph{Jacobi algebras}, or
\emph{generalized Poisson brackets}, or combinations or variations thereof; see,
for example, \cite[Chapter III, \S 5]{hermann}, \cite{ck},  \cite{kayg}, 
\cite[Chapter 5]{agore-m}, and references therein). In the particular case 
$[A,1] = 0$ this reduces to the classical notion of the \emph{Poisson bracket} 
on $A$. The following well known construction is a paradigmatic example of a 
contact bracket that is not a Poisson bracket: given a derivation $D$ of an
algebra $A$, consider the Lie bracket 
\begin{equation}\label{eq-ad}
[a,b] = D(a)b - D(b)a .
\end{equation}

The following question was posed as \cite[Question 1]{martinez-zelm}, and will 
be referred to in the sequel as the Mart\'inez--Zelmanov question: given two
commutative associative algebras $A$, $B$, a Poisson bracket $\liebrack_A$ on 
$A$, and a contact bracket $\liebrack_B$ on $B$, is it always possible to define
a contact bracket on the tensor product $A \otimes B$, extending the brackets 
$\liebrack_A$ and $\liebrack_B$? It is the purpose of the present note to answer
this question. The answer, in a sense, is both ``yes'' and ``no''. In general, the answer is negative, the corresponding 
example is constructed in \S \ref{sec-exam}. In a sense, this example is quite 
artificial and ``degenerate'' -- the resulting tensor product is isomorphic to 
the algebra $K[x,y,z]/(x^2,y^2,z^2)$; we also indicate how one can produce other
examples of a similar sort. On the other hand, for the most ``interesting'' and ``natural'' contact brackets appearing in mechanics, 
differential geometry, and structure theory of Lie algebras -- the contact 
brackets defined on polynomial algebras, or on reduced polynomial algebras in 
the case of positive characteristic -- the answer is affirmative; this is briefly
discussed in \S \ref{sec-pol}.

\section{Contact brackets that are not extended to the tensor product}
\label{sec-exam}

Rewrite the condition (\ref{eq-abc}) in the form
\begin{equation}\label{eq-yoyo}
[ab,c] = [a,c]b + [b,c]a - [1,c]ab ,
\end{equation}
and denote by $\mathcal K^-(A)$, respectively by $\mathcal K^+(A)$, the vector
space of all anticommutative, respectively commutative, brackets 
$\liebrack: A \times A \to A$ satisfying the condition (\ref{eq-yoyo}) (without
requiring them to satisfy the Jacobi identity or any other additional 
condition). (Note that while for anticommutative brackets the conditions
(\ref{eq-abc}) and (\ref{eq-yoyo}) are equivalent, for commutative brackets they
are not).

\begin{proposition}\label{prop-k}
Let $A$, $B$ be commutative associative algebras, one of them is 
finite-dimensional. Then there is an embedding of vector spaces
$$
\mathcal K^-(A \otimes B) \hookrightarrow
      \mathcal K^-(A) \otimes \mathcal K^+(B) 
\>+\> \mathcal K^+(A) \otimes \mathcal K^-(B) .
$$
More precisely, any bracket from $\mathcal K^-(A\otimes B)$ can be represented 
as a (finite) sum $\sum_{i\in \mathbb I} f_i \otimes g_i$, where for each 
$i\in \mathbb I$, either 
$f_i \in \mathcal K^-(A)$ and $g_i \in \mathcal K^+(A)$, or 
$f_i \in \mathcal K^+(A)$ and $g_i \in \mathcal K^-(A)$. Moreover, the equality
\begin{equation}\label{eq-22}
\sum_{i \in \mathbb I}
\Big(f_i(a,a^{\prime\prime})a^\prime - f_i(a^\prime,a^{\prime\prime})a\Big) 
\otimes 
\Big(g_i(b, b^{\prime\prime})b^\prime - g_i(b^\prime,b^{\prime\prime})b\Big)
= 0 .
\end{equation}
holds for any $a,a^\prime,a^{\prime\prime} \in A$, 
$b,b^\prime,b^{\prime\prime} \in B$.
\end{proposition}

\begin{proof}
The proof utilizes a simple linear-algebraic technique employed by us earlier to
compute various structures on the tensor product of algebras in terms of the
tensor factors, see \cite{low} and \cite{compendium}. 

Due to the finite-dimensionality assumption, we have an isomorphism of vector
spaces 
$$
\Hom_K(A \otimes B\otimes A \otimes B, A \otimes B) \simeq
\Hom_K(A \otimes A, A) \otimes \Hom_K(B \otimes B, B) ,
$$
thus any bracket $\liebrack$ on $A\otimes B$ can be represented in the form
\begin{equation}\label{eq-br}
[a \otimes b, a^\prime \otimes b^\prime] = 
\sum_{i\in \mathbb I} f_i(a,a^\prime) \otimes g_i(b,b^\prime)
\end{equation}
for any $a, a^\prime \in A$, $b, b^\prime \in B$, where $f_i$ are some brackets
on $A$, $g_i$ are some brackets on $B$. The condition (\ref{eq-abc}) is then
equivalent to
\begin{multline}\label{eq-m}
\sum_{i \in \mathbb I}
f_i(aa^\prime, a^{\prime\prime}) \otimes g_i(bb^\prime, b^{\prime\prime}) 
- f_i(a,a^{\prime\prime})a^\prime \otimes g_i(b, b^{\prime\prime})b^\prime
- f_i(a^\prime,a^{\prime\prime})a \otimes g_i(b^\prime,b^{\prime\prime})b
\\
- f_i(a^{\prime\prime},1)aa^\prime \otimes g_i(b^{\prime\prime},1)bb^\prime
= 0
\end{multline}
for any $a,a^\prime,a^{\prime\prime} \in A$, 
$b,b^\prime,b^{\prime\prime} \in B$.

Substituting in (\ref{eq-m}) $b = b^\prime = 1$ yields:
\begin{equation}\label{eq-m1}
\sum_{i \in \mathbb I}
\Big(  f_i(aa^\prime, a^{\prime\prime}) 
     - f_i(a,a^{\prime\prime})a^\prime 
     - f_i(a^\prime,a^{\prime\prime})a\Big) \otimes g_i(1,b^{\prime\prime})
- f_i(a^{\prime\prime},1)aa^\prime \otimes g_i(b^{\prime\prime},1)
= 0 .
\end{equation}

The condition of anticommutativity of $\liebrack$ is equivalent to
$$
\sum_{i\in \mathbb I} 
  f_i(a,a^\prime) \otimes g_i(b,b^\prime)
+ f_i(a^\prime,a) \otimes g_i(b^\prime,b)
= 0 .
$$

Symmetrizing this equality with respect to $a, a^\prime$ yields
\begin{equation}\label{eq-1}
\sum_{i\in \mathbb I} 
(f_i(a,a^\prime) - f_i(a^\prime,a)) \otimes (g_i(b,b^\prime) - g_i(b^\prime,b))
= 0\phantom{.}
\end{equation}
and
\begin{equation}\label{eq-2}
\sum_{i\in \mathbb I} 
(f_i(a,a^\prime) + f_i(a^\prime,a)) \otimes (g_i(b^\prime,b) + g_i(b,b^\prime))
= 0 .
\end{equation}

Applying to the last two equalities \cite[Lemma 1.1]{low}, and observing that
the condition of commutativity and anticommutativity of the same bilinear map
entails zero map (characteristic is not $2$), we may partition the index set
$\mathbb I = \mathbb I_1 \cup \mathbb I_2$ such that $f_i$ is anticommutative 
and $g_i$ is commutative for $i\in \mathbb I_1$, and $f_i$ is commutative and 
$g_i$ is anticommutative for $i \in \mathbb I_2$\footnote{
Of course, this is just a manifestation of the vector space isomorphism
$C^2(A \otimes B) \simeq C^2(A) \otimes S^2(B) + S^2(A) \otimes C^2(B)$, where
$C^2$ and $S^2$ denote the vector space of anticommutative and commutative
brackets on the corresponding algebra, respectively. However, we use below a 
similar reasoning in more complicated situations, where the argument based on
the above decomposition of $C^2(A \otimes B)$ is not enough.
}.

Using this partition, the equality (\ref{eq-m1}) can be rewritten as
\begin{equation}\label{eq-33}
\sum_{i \in \mathbb I}
\Big(  f_i(aa^\prime, a^{\prime\prime}) 
     - f_i(a,a^{\prime\prime})a^\prime 
     - f_i(a^\prime,a^{\prime\prime})a 
     + f_i(1,a^{\prime\prime})aa^\prime\Big) \otimes g_i(1,b^{\prime\prime})
= 0 .
\end{equation}

Symmetrizing the equality (\ref{eq-m}) with respect to $a, a^\prime$, we get the
equality (\ref{eq-22}).

Applying \cite[Lemma 1.1]{low} to the equalities (\ref{eq-33}) and 
(\ref{eq-22}), we get a partition
$\mathbb I = 
\mathbb I_{11} \cup \mathbb I_{12} \cup \mathbb I_{21} \cup \mathbb I_{22}$ such
that
\begin{gather*}
  f_i(aa^\prime, a^{\prime\prime}) 
- f_i(a,a^{\prime\prime})a^\prime 
- f_i(a^\prime,a^{\prime\prime})a 
+ f_i(1,a^{\prime\prime})aa^\prime = 0, \>
f_i(a,a^{\prime\prime})a^\prime = f_i(a^\prime,a^{\prime\prime})a 
\quad\text{if}\> i \in \mathbb I_{11}
\\
  f_i(aa^\prime, a^{\prime\prime}) 
- f_i(a,a^{\prime\prime})a^\prime 
- f_i(a^\prime,a^{\prime\prime})a 
+ f_i(1,a^{\prime\prime})aa^\prime = 0, \>
g_i(b, b^{\prime\prime})b^\prime = g_i(b^\prime,b^{\prime\prime})b
\quad\text{if } i \in \mathbb I_{12}
\\
f_i(a,a^{\prime\prime})a^\prime = f_i(a^\prime,a^{\prime\prime})a, \>
g_i(1,b^{\prime\prime}) = 0
\quad\text{if } i \in \mathbb I_{21}
\\
g_i(1,b^{\prime\prime}) = 0, \>
g_i(b, b^{\prime\prime})b^\prime = g_i(b^\prime,b^{\prime\prime})b
\quad\text{if } i \in \mathbb I_{22} .
\end{gather*}

Observe that the identity $\varphi(x,z)y = \varphi(y,z)x$ for a bracket 
$\varphi$ on an algebra with unit is equivalent to the identity
$\varphi(x,y) = \varphi(1,y)x$. Consequently, the second condition for the
brackets $f_i$ with $i\in \mathbb I_{11}$ implies the first one, $g_i$ with 
$i\in \mathbb I_{22}$ vanish, and thus the identity
\begin{equation}\label{eq-ka}
  f_i(aa^\prime, a^{\prime\prime}) 
- f_i(a,a^{\prime\prime})a^\prime 
- f_i(a^\prime,a^{\prime\prime})a 
+ f_i(1,a^{\prime\prime})aa^\prime = 0
\end{equation}
holds for any $i\in \mathbb I$.

Similarly, the identity
\begin{equation}\label{eq-kb}
  g_i(bb^\prime, b^{\prime\prime}) 
- g_i(b,b^{\prime\prime})b^\prime 
- g_i(b^\prime,b^{\prime\prime})b
+ g_i(1,b^{\prime\prime})bb^\prime = 0
\end{equation}
also holds for any $i\in \mathbb I$.

Applying again \cite[Lemma 1.1]{low} to the equalities (\ref{eq-1}) and 
(\ref{eq-2}), we get again the partition 
$\mathbb I = \mathbb I_1 \cup \mathbb I_2$ such that $f_i$ is anticommutative 
and $g_i$ is commutative for $i\in \mathbb I_1$, and $f_i$ is commutative and 
$g_i$ is anticommutative for $i \in \mathbb I_2$, what, together with 
(\ref{eq-ka}) and (\ref{eq-kb}), gives decomposition of the bracket 
(\ref{eq-br}) exactly as in the statement of the theorem.
\end{proof}

One might be tempted to try to pursue these reasonings further in an attempt to
establish a formula expressing $\mathcal K^-(A\otimes B)$ in terms of certain
invariants of $A$ and $B$, similarly how it is done for various structures on 
the tensor product of two algebras in \cite{low} and \cite{compendium}. However,
this does not seem to be possible. Indeed, for this approach to succeed, all the
brackets from $\mathcal K^-(A \otimes B)$ should be, at the end, representable 
as the sum of decomposable ones, i.e., the brackets of the form $f \otimes g$, where $f$ is
a bracket on $A$, and $g$ is a bracket on $B$; but a glance, for example, at the
brackets of type (\ref{eq-ad}), or the more complicated brackets defining the 
simple Lie algebras of contact type (both infinite dimensional, and finite 
dimensional in positive characteristic; for the most general brackets of this 
type, see \S \ref{sec-pol}), defies such a possibility. In some particular 
cases, however, it is possible to get such exact formulas -- for example, in 
Proposition \ref{prop-x2} below. Note also that if we assume in (\ref{eq-yoyo})
$[1,A] = 0$ (thus getting the condition defining Poisson brackets), the 
situation becomes much more tractable, even without assuming commutativity of 
algebras and anticommutativity of the bracket; see, for example, \cite{eremita}
for a sample of possible results in this direction.

\medskip

Now we start to construct an example providing a negative answer to the
Mart\'inez--Zelmanov question. At the end, our example turns out to be 
$K[x,y,z]/(x^2,y^2,z^2)$, a quite trivial $8$-dimensional algebra; the fact that
it provides a negative answer to the question could be established by trivial, 
if a bit tedious, calculations. However, we choose to establish it as a 
corollary of intermediate statements formulated in a greater generality; this 
will allow us to understand better the structure of contact brackets on the 
tensor product in terms of the tensor factors, and opens possibilities to 
construct further examples and counter-examples of contact brackets with desired 
properties.

By \emph{generalized derivation} of a commutative associative algebra $A$, we 
will understand a linear map $D: A \to A$ such that 
\begin{equation}\label{eq-gd}
D(ab) = D(a)b + D(b)a - D(1)ab
\end{equation}
for any $a,b \in A$. (Generalized derivations of associative rings not 
necessarily commutative, and not necessarily having a unit -- in the latter case
$D(1)$ in the formula (\ref{eq-gd}) is replaced by an arbitrary fixed element of
the ring -- were studied in a number of papers, see, for example, 
\cite{nakajima}).

The vector space of all generalized derivations of $A$ will
be denoted by $\GDer(A)$. Particular cases of generalized derivations are the 
usual derivations (with $D(1) = 0$), and multiplications $R_u(a) = au$ on a 
fixed element $u \in A$. Thus we always have an inclusion of vector spaces
\begin{equation}\label{eq-gder}
\Der(A) \oplus A \subseteq \GDer(A) ,
\end{equation}
where $\Der(A)$ denotes the vector space (actually, a Lie algebra) of (ordinary)
derivations of $A$, and the second direct summand at the left-hand side, $A$, is
identified with the vector space of all multiplications by elements of $A$, via 
$u \leftrightarrow R_u$.

A commutative associative algebra $A$ will be called \emph{univariate-like}, if
it satisfies the following conditions:
\begin{enumerate}[\upshape(a)]
\item 
The inclusion 
(\ref{eq-gder}) is an equality, i.e., each generalized derivation of $A$ is a 
sum of a derivation and a multiplication by an element of $A$;

\item
$\Der(A)$ is an one-generated free $I_A$-module for some ideal $I_A$ of $A$, 
i.e., $\Der(A) = I_A\partial_A$ for some linear map $\partial_A: A \to A$ (note
that $\partial_A$ does not have to be a derivation of $A$);

\item
The preimage of $1$ under the free generator $\partial_A$ is not empty.
\end{enumerate}

Paradigmatic examples of univariate-like algebras are provided by the following

\begin{lemma}\label{l-kx}
The algebras $K[x]$ and $K[x]/(x^n)$ for any $n\in \mathbb N$, are 
univariate-like.
\end{lemma}

\begin{proof}
This is well known (and could be established by straightforward calculations). 
Note that the properties of the algebra $K[x]/(x^n)$ are quite different 
depending on whether $n$ is divided by the characteristic of the ground field or
not (in the latter case each derivation is of the form 
$f \frac{\dcobound}{\dcobound x}$, where $f \in xK[x]/(x^n)$), but the conclusion is valid in both cases.
\end{proof}

\begin{lemma}\label{prop-der}
Let $A$ be an univariate-like commutative associative algebra. Then there are 
the following isomorphisms of vector spaces:
\begin{enumerate}[\upshape(i)]
\item
$\mathcal K^-(A) \simeq \Der(A) \simeq I_A$; each bracket from $\mathcal K^-(A)$
is of the form (\ref{eq-ad}), i.e.,
$$
[a,b] = u\big(a\partial_A(b) - b\partial_A(a)\big)
$$
for some $u\in I_A$.

\item 
$\mathcal K^+(A) \simeq I_A \oplus \GDer(A) \simeq I_A \oplus I_A \oplus A$; 
each bracket from $\mathcal K^+(A)$ is of the form
$$
[a,b] = u\partial_A(a)\partial_A(b) + v\partial_A(ab) + abw ,
$$
for some $u,v \in I_A$, $w \in A$.
\end{enumerate}
\end{lemma}

\begin{proof}
Let $\liebrack$ be a bracket on $A$ satisfying the identity (\ref{eq-yoyo}). 
Then for each $c \in A$, the linear map $a \mapsto [a,c]$ is a generalized 
derivation of $A$. Hence 
$$
[a,c] = \varphi(c)\partial(a) + a\psi(c)
$$
for some maps $\varphi: A \to I_A$, $\psi: A \to A$, and any $a,c\in A$. 
Obviously, $\varphi, \psi$ can be chosen to be linear.

Also, according to the condition (c), fix $x\in A$ such that 
$\partial_A(x) = 1$.

(i)
If $\liebrack$ is anticommutative, then 
$$
\varphi(c)\partial_A(a) + a\psi(c) + \varphi(a)\partial_A(c) + c\psi(a) = 0 .
$$
Substituting here $1$'s yields $\psi(a) = -\varphi(1)\partial_A(a)$, and hence
$$
  \big(\varphi(c) - c\varphi(1)\big)\partial_A(a) 
+ \big(\varphi(a) - a\varphi(1)\big)\partial_A(c) = 0
$$
for any $a,c \in A$. Substituting here $x$ instead of $a$ and $c$, we get 
$\varphi(x) = x\varphi(1)$, $\varphi(a) = a\varphi(1)$, and, finally,
$$
[a,c] = \varphi(1)\big(c\partial_A(a) - a\partial_A(c)\big) .
$$

(ii) 
If $\liebrack$ is commutative, then 
$$
\varphi(c)\partial_A(a) + a\psi(c) = \varphi(a)\partial_A(c) + c\psi(a) .
$$
Substituting here $c=1$ yields $\psi(a) = \varphi(1)\partial_A(a) + a\psi(1)$, 
and hence
$$
\big(\varphi(c) - c\varphi(1)\big)\partial_A(a) =
\big(\varphi(a) - a\varphi(1)\big)\partial_A(c)
$$
for any $a,c \in A$. Substituting here $c=x$ yields
$$
\varphi(a) = \big(\varphi(x) - x\varphi(1)\big)\partial_A(a) + a\varphi(1) ,
$$
and hence
$$
[a,c] = (\varphi(x) - x\varphi(1))\partial_A(a)\partial_A(c) 
+ \varphi(1)\partial_A(ac) + ac\psi(1) .
$$
\end{proof}

\begin{proposition}\label{prop-x2}
For any commutative associative algebra $A$, 
$$
\mathcal K^-\big(A \otimes K[x]/(x^2)\big) \simeq 
\mathcal K^-(A) \oplus \mathcal K^-(A) \oplus \Der(A) \oplus A .
$$
Each bracket from $\mathcal K^-\big(A \otimes K[x]/(x^2)\big)$ is of the form
\begin{alignat}{3}\label{eq-1ab}
&[a \otimes 1, b \otimes 1] \>&=&\> \alpha(a,b) \otimes 1 + \beta(a,b) \otimes x
\notag \\
&[a \otimes 1, b \otimes x] \>&=&\> 
                                   \big(\alpha(a,b) + bD(a) + abu\big) \otimes x
\\
&[a \otimes x, b \otimes x] \>&=&\> 0 
\notag
\end{alignat}
for any $a,b \in A$, where $\alpha, \beta \in \mathcal K^-(A)$, $D \in \Der(A)$,
$u\in A$.
\end{proposition}

\begin{proof}
By Lemmas \ref{l-kx} and \ref{prop-der}, $\mathcal K^-\big(K[x]/(x^2)\big)$ is 
$1$-dimensional, linearly spanned by the bracket 
$$
\varphi(f,g) = 
x\big(f\frac{\dcobound}{\dcobound x}(g) - g\frac{\dcobound}{\dcobound x}(f)\big)
,
$$
and $\mathcal K^+\big(K[x]/(x^2)\big)$ is $4$-dimensional, with the basis 
$$
\varphi_1(f,g) = 
x\frac{\dcobound }{\dcobound x}(f)\frac{\dcobound}{\dcobound x}(g), \quad
\varphi_2(f,g) = x\frac{\dcobound}{\dcobound x}(fg),                \quad   
\varphi_3(f,g) = fg - x\frac{\dcobound}{\dcobound x}(fg),           \quad
\varphi_4(f,g) = xfg ,
$$
where $f,g \in K[x]/(x^2)$. We have:
$$
\varphi(1,x) = -\varphi(x,1) = x, \quad
\varphi_1(x,x) = x, \quad 
\varphi_2(1,x) = \varphi_2(x,1) = x, \quad
\varphi_3(1,1) = 1, \quad
\varphi_4(1,1) = x ,
$$
and the values on all other pairs of the monomials $1,x$, are zero.

By Proposition \ref{prop-k}, each bracket from 
$\mathcal K^-\big(A \otimes K[x]/(x^2)\big)$ can be represented as
$$
[a \otimes f, b \otimes g] = \chi(a,b) \otimes \varphi(f,g)+ 
\sum_{i=1}^4 \chi_i(a,b) \otimes \varphi_i(f,g)
$$
where $\chi \in \mathcal K^+(A)$, and $\chi_i \in \mathcal K^-(A)$, 
$i=1,2,3,4$. 

Writing the equality (\ref{eq-yoyo}) for triple $a \otimes 1$, $b \otimes x$, 
$c \otimes 1$, and utilizing the fact that $\chi \in \mathcal K^+(A)$, 
$\chi_2 \in \mathcal K^-(A)$, we get:
\begin{equation*}
- b\chi(a,c) + ab\chi(1,c)
+ b\chi_2(a,c) - ab\chi_2(1,c)
= b \chi_3(a,c) - ab \chi_3(1,c)
\end{equation*}
for any $a,b,c \in A$.

Substitute here $b=1$:
\begin{equation}\label{eq-ac}
- \chi(a,c) + a\chi(1,c)
+ \chi_2(a,c) - a\chi_2(1,c)
= \chi_3(a,c) - a \chi_3(1,c) .
\end{equation}

Substituting into (\ref{eq-ac}), in its turn, $c=1$, using skew-symmetry of 
$\chi_2$ and $\chi_3$, and substituting back the obtained equality into 
(\ref{eq-ac}), we get
\begin{equation}\label{eq-chi2}
\chi_2(a,c) = \chi_3(a,c) + \chi(a,c) - 2a\chi(1,c) + ac\chi(1,1) .
\end{equation}

Symmetrizing this equality with respect to $a,c$, we get
\begin{equation}\label{eq-chi}
\chi(a,c) = a\chi(1,c) + c\chi(1,a) - ac\chi(1,1) .
\end{equation}

Substitute (\ref{eq-chi}) back into (\ref{eq-chi2}):
$$
\chi_2(a,c) = \chi_3(a,c) + c\chi(1,a) - a\chi(1,c)  .
$$

Taking into account (\ref{eq-chi}), the condition $\chi \in \mathcal K^+(A)$ is
equivalent to
$$
\chi(1,ab) = b\chi(1,a) + a\chi(1,b) - ab\chi(1,1) ,
$$
i.e., $\chi(1,\cdot) \in \GDer(A)$.

Denoting $D(a) = 2\big(\chi(1,a) - \chi(1,1)a\big)$ and $u = \chi(1,1)$, the
obtained formulas can be rewritten as: $D \in \Der(A)$, and
\begin{alignat*}{3}
&\chi(a,b)   \>&=&\> \frac 12\big(aD(b) + bD(a)\big) + abu \\
&\chi_2(a,b) \>&=&\> \chi_3(a,b) + \frac 12\big(bD(a) - aD(b)\big) .
\end{alignat*}

Writing the equality (\ref{eq-yoyo}) for triple $1 \otimes 1$, $b \otimes x$, 
$c \otimes x$, we get $\chi_1(b,c) = 0$.

Summarizing, we get that the bracket $\liebrack$ is represented in the form 
(\ref{eq-1ab}) (where $\alpha = \chi_3$ and $\beta = \chi_4$). Conversely, it is
trivial to verify that each such skew-symmetric bracket lies in 
$\mathcal K^-\big(A \otimes K[x]/(x^2)\big)$.
\end{proof}

\begin{corollary}\label{cor-1}
For any commutative associative algebra $A$, each contact bracket on 
$A \otimes K[x]/(x^2)$ is of the form (\ref{eq-1ab}), subject to additional 
conditions
\begin{multline*}
  \beta(\alpha(a,b),c) + \beta(\alpha(b,c),a) + \beta(\alpha(c,a),b) 
+ \alpha(\beta(a,b),c) + \alpha(\beta(b,c),a) + \alpha(\beta(c,a),b) 
\\
- \beta(a,b)D(c) - \beta(b,c)D(a) - \beta(c,a)D(b) 
\\
- \big(c\beta(a,b) + a\beta(b,c) + b\beta(c,a)\big) u
= 0
\end{multline*}
and
\begin{multline*}
  D(\alpha(a,b)) 
- \alpha(D(a),b) + \alpha(D(b),a) + \alpha(1,b)D(a) - \alpha(1,a)D(b)
\\
+ \alpha(u,a)b - \alpha(u,b)a - \alpha(a,b)u + 2\alpha(1,b)au - 2\alpha(1,a)bu
= 0
\end{multline*}
for any $a,b,c \in A$.
\end{corollary}

\begin{proof}
Amounts to writing the Jacobi identity for the bracket (\ref{eq-1ab}), which is 
equivalent to the specified equalities.
\end{proof}

\begin{corollary}\label{cor-2}
Let $A$ be a commutative associative algebra, and 
$\liebrack_A \in \mathcal K^-(A)$. Then any bracket from 
$\mathcal K^-\big(A \otimes K[x]/(x^2)\big)$ extending the bracket $\liebrack_A$
on $A$, and the bracket $\varphi$ on $K[x]/(x^2)$, is of the form (\ref{eq-1ab})
such that $\alpha = \liebrack_A$, $\beta = 0$, and $u=1$.
\end{corollary}

\begin{proof}
Amounts to substituting the equalities 
$[a \otimes 1, b \otimes 1] = [a,b]_A \otimes 1$ and 
$[1 \otimes f, 1 \otimes g] = 1 \otimes \varphi(f,g)$ in (\ref{eq-1ab}).
\end{proof}

Now we are ready to prove the main result of this note, answering in negative
the Mart\'inez--Zelmanov question.

\begin{theorem}\label{th-1}
There exist two commutative associative algebras $A$, $B$, a Poisson bracket 
$\liebrack_A$ on $A$, and a contact bracket $\liebrack_B$ on $B$, such that 
there is no contact bracket on $A \otimes B$ extending brackets $\liebrack_A$ 
and $\liebrack_B$.
\end{theorem}

\begin{proof}
Let $B = K[x]/(x^2)$ and $\liebrack_B = \varphi$. Assume $\liebrack$ is a 
contact bracket on $A \otimes K[x]/(x^2)$, extending the brackets $\liebrack_A$
on $A$, and $\varphi$ on $K[x]/(x^2)$. By Corollaries \ref{cor-1} and 
\ref{cor-2}, 
\begin{alignat*}{3}
&[a \otimes 1, b \otimes 1] \>&=&\> [a,b]_A \otimes 1 \\
&[a \otimes 1, b \otimes x] \>&=&\> \big([a,b]_A + bD(a) + ab\big) \otimes x
\\
&[a \otimes x, b \otimes x] \>&=&\> 0 , 
\end{alignat*}
where $D \in \Der(A)$ is such that
\begin{multline*}
  D([a,b]_A) 
- [D(a),b]_A + [D(b),a]_A + [1,b]_A D(a) - [1,a]_A D(b)
\\
+ [1,a]_A b - [1,b]_A a - [a,b]_A + 2[1,b]_A a - 2[1,a]_A b
= 0
\end{multline*}
for any $a,b \in A$. If $\liebrack_A$ is a Poisson bracket, the last equality is
reduced to
\begin{equation}\label{eq-inv}
D([a,b]_A) - [D(a),b]_A + [D(b),a]_A = [a,b]_A .
\end{equation}

The left-hand side of this equality coincides with the standard action of $D$ on
the space of bilinear maps on $A$, so our task reduces to finding an algebra $A$
with a Poisson bracket $\liebrack_A$ not invariant with respect to this action 
for any derivation $D$ of $A$. There seems to be a plethora of such examples, 
one of the easiest is provided by the algebra $A = K[x,y]/(x^2,y^2)$ and the 
Poisson bracket $[x,y]_A = xy$.

Indeed, any derivation $D$ of $K[x,y]/(x^2,y^2)$ is of the form 
$f \frac{\dcobound}{\dcobound x} + g \frac{\dcobound}{\dcobound y}$, where 
$f \in Kx \oplus Kxy$ and $g \in Ky \oplus Kxy$, so substituting $x=a$ and $y=b$
in (\ref{eq-inv}) yields
\begin{equation}\label{eq-fg}
fy + gx - [f,y] + [g,x]
\end{equation}
at the left-hand side, and $xy$ at the right-hand side. But a quick check for 
all possibilities for $f,g$ shows that (\ref{eq-fg}) always vanishes.
\end{proof}

Our scheme allows to construct further examples providing a negative answer to
the Mart\'inez--Zelmanov question, and to its modifications, for example, when 
both brackets on the tensor factors are contact. This is left as an exercise to
the reader.

Further, more elaborated, examples quite possibly could be obtained by 
considering quotient of polynomial algebras by non-homogeneous ideals. Such 
quotients may possess quite involved Poisson brackets, see, for example, 
\cite{kubo} and references therein.

Finally, let us mention another series of examples. These examples are brackets
of the form
\begin{equation}\label{eq-t}
[a \otimes b, a^\prime \otimes b^\prime] = [a,a^\prime]_A \otimes bb^\prime +
                                           aa^\prime \otimes [b,b^\prime]_B ,
\end{equation}
where $a,a^\prime \in A$, $b,b^\prime \in B$. If $\liebrack_A$ and $\liebrack_B$
are Poisson brackets on $A$ and $B$ respectively, then this is again a Poisson 
bracket on $A \otimes B$, a classical construction known from the literature as
the tensor product of two Poisson brackets.

What happens if $\liebrack_A$, $\liebrack_B$ are contact brackets? It is 
straightforward to check that in this case the bracket (\ref{eq-t}) satisfies 
the condition (\ref{eq-abc}), so the question reduces to whether (\ref{eq-t}) 
satisfies the Jacobi identity or not. It turns out that this is no longer 
necessarily true if at least one of $\liebrack_A$, $\liebrack_B$ is a contact 
bracket\footnote{
Note that Remark 5.1.1(2) in \cite{agore-m}, which claims the contrary, is 
incorrect.
}. 
The corresponding examples were considered in an old interesting survey 
\cite{kostr-dzhu}.

In the examples considered in that survey, $A$ and $B$ are reduced polynomial 
algebras of the form \linebreak $K[x_1,\dots,x_n]/(x_1^p,\dots,x_n^p)$, defined
over a field of characteristic $p>0$, with corresponding brackets yielding 
simple Lie algebras of Cartan type of the series $W_n$, $H_{2n}$, $K_{2n+1}$ (or
algebras close to them). For example, it was observed in \cite[\S 3]{kostr-dzhu}
that the bracket (\ref{eq-t}) defined on $W_1 \otimes W_1$ satisfies the Jacobi
identity, while on $W_1 \otimes H_2$, $H_2 \otimes K_3$, and $K_3 \otimes K_3$, it does not. 
This still does not give a full answer to the Mart\'inez--Zelmanov question, as
it could happen that the brackets on $A$ and $B$ can be extended to 
$A \otimes B$ in a way different from (\ref{eq-t}). In fact, as briefly 
explained in the next section, this is always the case: any contact bracket on 
two (reduced) polynomial algebras could be extended to a contact bracket on 
their tensor product.

\section{Contact brackets that are extended to the tensor product}
\label{sec-pol}

Most of the content of this section is hardly new: it is either contained in the
literature, implicitly or explicitly, or can be obtained by immediate analogy with the known (Poisson) case; we omit almost all of the proofs. Still,
the final conclusion -- that any contact bracket on two polynomial algebras 
could be extended to a contact bracket on their tensor product -- seems to be 
not explicitly recorded in the literature, and provides a nice contrast with 
Theorem~\ref{th-1}.

When discussing brackets on polynomial and close to them algebras, it is 
convenient to adopt the following shorthand notation. For two linear operators 
$D, F: A \to A$ on an algebra $A$, its exterior product 
$D \wedge F: A \times A \to A$ is defined as
$$
(D \wedge F)(a,b) = D(a)F(b) - D(b)F(a) ,
$$
where $a,b \in A$. For example, using this notation, the bracket (\ref{eq-ad}) 
can be written as $D \wedge \id_A$, where $\id_A$ is the identity map on $A$. 
More generally, consider the bracket of the form
\begin{equation}\label{eq-k}
\liebrack = \sum_{i=1}^n (D_i \wedge F_i) + D \wedge \id_A ,
\end{equation}
where $D, D_1, \dots, D_n, F_1, \dots, F_n \in \Der(A)$. It is easy to check 
that each such bracket belongs to $\mathcal K^-(A)$. Let us call an associative
commutative algebra $A$ \emph{standard}, if, conversely, each bracket from 
$\mathcal K^-(A)$ is of the form (\ref{eq-k}). In particular, 
Lemma~\ref{prop-der}(i) implies that an univariate-like algebra is standard. We
also have

\begin{proposition}\label{prop-p}
The polynomial algebra $K[x_1,\dots,x_n]$, and the reduced polynomial algebra
\linebreak
$K[x_1,\dots,x_n]/(x_1^p,\dots,x_n^p)$ are standard.
\end{proposition}

\begin{proof}
More precisely, any bracket from $\mathcal K^-(K[x_1,\dots,x_n])$ is of the form
\begin{equation*}
\liebrack = 
\sum_{1 \le i < j \le n} \Big(f_{ij} 
\frac{\dcobound}{\dcobound x_i} \wedge \frac{\dcobound}{\dcobound x_j} \Big)
+ \Big(\sum_{i=1}^n f_i \frac{\dcobound}{\dcobound x_i}\Big) \wedge 
  \id_{K[x_1,\dots,x_n]}
\end{equation*}
for some elements $f_{ij}, f_i \in K[x_1,\dots,x_n]$. An elementary proof goes 
similarly to the classical Poisson case (where all $f_i$'s vanish); see, for 
example, \cite[Proposition 1.6]{poisson-book}. We set $f_i = [1,x_i]$ and 
$f_{ij} = [x_i,x_j] + f_ix_j - f_jx_i$, and then proceed by induction on the sum of degrees of the monomials $a, b$ in the expression $[a,b]$.

The reasoning in the case of reduced polynomial algebra is the same.
\end{proof}

Necessary and sufficient conditions for the bracket (\ref{eq-k}) to be a 
contact bracket, are
\begin{align}\label{eq-sch}
\begin{split}
&\llbracket \mathcal D, D \rrbracket \>=\> 0  \\
&\llbracket \mathcal D, \mathcal D \rrbracket \>=\> 2\, \mathcal D \wedge D ,
\end{split}
\end{align}
where $\mathcal D = \sum_{i=1}^n D_i \wedge F_i$, and 
$\llbracket \cdot,\cdot \rrbracket$ is the Schouten bracket 
(\cite[Erratum]{kirillov}, see also \cite[Lemma 3.7]{ck}; note in passing that
the conditions of Corollary \ref{cor-1} are reminiscent of the conditions 
(\ref{eq-sch}), as the left-hand sides of equalities there are reminiscent of 
the respective Schouten brackets, but since the algebra $A$ there is, generally,
not standard, the conditions of Corollary \ref{cor-1} look a bit more involved).

Now we can easily establish

\begin{theorem}\label{th-2}
Let $A, B$ be two standard commutative associative algebras, $\liebrack_A$ a
contact bracket on $A$, and $\liebrack_B$ a contact bracket on $B$. Then there
exists a contact bracket on $A \otimes B$ extending brackets $\liebrack_A$ and
$\liebrack_B$.
\end{theorem}

\begin{proof}
We have
\begin{gather*}
\liebrack_A = \sum_{i=1}^n (D_i \wedge D_i^\prime) + D \wedge \id_A \\ 
\liebrack_B = \sum_{i=1}^n (F_i \wedge F_i^\prime) + F \wedge \id_B 
\end{gather*}
for some $D,D_i,D_i^\prime \in \Der(A)$ and $F,F_i,F_i^\prime \in \Der(B)$.

Let us define the bracket $\liebrack$ on $A\otimes B$ as
$$
\liebrack = \sum_{i=1}^n \Big(
  (D_i \otimes \id_B) \wedge (D_i^\prime \otimes \id_B)
+ (\id_A \otimes F_i) \wedge (\id_A \otimes F_i^\prime)
\Big) 
+ (D \wedge \id_B) \otimes m_B + m_A \otimes (\id_A \wedge F) ,
$$
where $m_A: A \times A \to A$ and $m_B: B \times B \to B$ are multiplications in
algebras $A$ and $B$, respectively. It is a matter of routine verification, 
using the equalities (\ref{eq-sch}) for the brackets $\liebrack_A$ and 
$\liebrack_B$, to establish the same equalities for the bracket $\liebrack$.
\end{proof}

Proposition \ref{prop-p} implies that the conclusion of Theorem \ref{th-2} is 
applicable to the tensor product of two (reduced) polynomial algebras.

\section*{Acknowledgements}

Thanks are due to Ivan Kaygorodov for pointing out some relevant literature. 
GAP \cite{gap} was used to check some of the computations performed in this 
note.


\begin{thebibliography}{LPV}

\bibitem[AM]{agore-m} A.L. Agore and G. Militaru, 
\emph{Extending Structures. Fundamentals and Applications}, CRC Press, 2020.

\bibitem[CK]{ck} N. Cantarini and V. Kac,
\emph{Classification of linearly compact simple Jordan and generalized Poisson superalgebras},
J. Algebra \textbf{313} (2007), no.1, 100--124.

\bibitem[E]{eremita} D. Eremita, 
\emph{Biderivations on tensor products of algebras},
Comm. Algebra \textbf{46} (2018), no.4, 1722--1726.

\bibitem[G]{gap} 
The GAP Group, \emph{GAP -- Groups, Algorithms, and Programming}, 
Version 4.10.2, 2019; \newline
\texttt{https://www.gap-system.org/}

\bibitem[H]{hermann} R. Hermann, 
\emph{Yang-Mills, Kaluza-Klein, and the Einstein Program}, 
Interdisciplinary Mathematics, Vol. XIX, Math Sci Press, 1978.

\bibitem[Ka]{kayg} I. Kaygorodov, 
\emph{Algebras of Jordan brackets and generalized Poisson algebras},
Lin. Multilin. Algebra \textbf{65} (2017), no.6, 1142--1157.

\bibitem[Ki]{kirillov} A.A. Kirillov, \emph{Local Lie algebras}, 
Uspekhi Mat. Nauk \textbf{31} (1976), no.4, 57--76; 
Erratum: \textbf{32} (1977), no.1, 267 (in Russian); 
Russ. Math. Surv. \textbf{31} (1976), no.4, 55--75 (English translation).

\bibitem[KD]{kostr-dzhu} A.I. Kostrikin and A.S. Dzhumadil'daev, \emph{Modular Lie algebras: new trends}, Algebra. Proceeding of the International Algebraic Conference on the Occasion of the 90th Birthday of A.G. Kurosh (ed. Yu. Bahturin), 
De Gruyter, 2000, 181--203. 

\bibitem[Ku]{kubo} F. Kubo, 
\emph{Lie structures on $\mathfrak k[x_1,\dots,x_n,y]/(y^3-3py-q)$}, 
Bull. Kyushu Inst. Tech. Math. Natur. Sci. \textbf{35} (1988), 1--6.

\bibitem[LPV]{poisson-book} 
C. Laurent-Gengoux, A. Pichereau, and P. Vanhaecke, 
\emph{Poisson Structures}, Springer, 2013.

\bibitem[MZ]{martinez-zelm} C. Mart\'inez and E. Zelmanov, 
\emph{Brackets, superalgebras and spectral gap},
S\~ao Paulo J. Math. Sci. \textbf{13} (2019), no.1, 112--132.

\bibitem[N]{nakajima} A. Nakajima,
\emph{On categorical properties of generalized derivations},
Scientiae Mathematicae \textbf{2} (1999), no.3, 345--352.

\bibitem[Z1]{low} P. Zusmanovich,
\emph{Low-dimensional cohomology of current Lie algebras and analogs of the 
Riemann tensor for loop manifolds}, 
Lin. Algebra Appl. \textbf{407} (2005), 71--104; arXiv:math/0302334.

\bibitem[Z2]{compendium} \bysame, 
\emph{A compendium of Lie structures on tensor products},
Zapiski Nauchnykh Seminarov POMI \textbf{414} (2013) (N.A. Vavilov Festschrift),
40--81; reprinted in J. Math. Sci. \textbf{199} (2014), no.3, 266--288;
arXiv:1303.3231.

\end{thebibliography}
\end{document}